\documentclass[12pt]{amsart}
\usepackage{a4wide}
\usepackage{amssymb}
\usepackage{amsthm}
\usepackage{amsmath}
\usepackage{amscd}
\usepackage[all]{xy}
\usepackage{verbatim}
\usepackage{mathrsfs}
\usepackage{dsfont}
\numberwithin{equation}{section}

\theoremstyle{plain}
\newtheorem{theorem}{Theorem}[section]
\newtheorem{corollary}[theorem]{Corollary}
\newtheorem{lemma}[theorem]{Lemma}
\newtheorem{proposition}[theorem]{Proposition}

\theoremstyle{definition}
\newtheorem{definition}[theorem]{Definition}

\newtheorem{example}[theorem]{Example}

\theoremstyle{remark}

\newcommand{\Q}{\mathbb{Q}}
\newcommand{\Z}{\mathbb{Z}}

\newcommand{\C}{\mathbb{C}}

\newcommand{\vol}{\operatorname{vol}}
\newcommand{\tr}{\operatorname{tr}}

\newcommand{\Gal}{\operatorname{Gal}}

\newcommand{\ff}{\operatorname{ if }}

\newcommand{\CM}{\mathcal{CM}}

\newcommand{\Ind}{\operatorname{Ind}}

\newcommand{\OO}{\mathcal O}

\font\cute=cmitt10 at 12pt 
\newcommand{\kay}{{\text{\cute k}}}

\newcommand{\Fal}{\operatorname{Fal}}

\begin{document}


\title[CM fields of Dihedral type and the Colmez conjecture]{CM fields of Dihedral type and  the Colmez conjecture}
\author[ Tonghai Yang and Hongbo Yin]{Tonghai Yang and Hongbo Yin}

\address{Department of Mathematics, University of Wisconsin Madison,
Van Vleck Hall, Madison, WI 53706, USA}
\email{thyang@math.wisc.edu}
\address{Academy of Mathematics and Systems Science, Morningside center of Mathematics, Chinese Academy of Sciences, Beijing 100190}
\email{yinhb@math.ac.cn}


\begin{abstract}  In this paper, we consider some CM fields which we call of dihedral type and compute the Artin $L$-functions associated to all CM types of these CM fields. As a consequence of this calculation, we see that the Colmez conjecture in this case is very closely related to understanding the log derivatives of certain Hecke characters of real quadratic fields. Recall that the `abelian case' of the Colmez conjecture, proved by Colmez himself, amounts to understanding the log derivatives of  Hecke characters of $\Q$ (cyclotomic characters). In this paper, we also prove that the Colmez conjecture holds for `unitary CM types of signature $(n-1, 1)$' and  holds on average for `unitary CM types of a fixed CM number field of  signature $(n-r, r)$'.
\end{abstract}


\maketitle

\section{Introduction}

In his influential work in 1993 \cite{Colmez}, Colmez discovered a conjectural deep and precise relation between the geometry (Faltings height) and analysis (log derivatives of $L$-functions)   of CM types, which is a vast generalization of the well-known Chowla-Selberg formula and Kronecker limit formula. The average version was recently used by Tsimerman to prove the Andre-Oort conjecture \cite{Tsimerman}. The average Colmez conjecture was then proved independently by Yuan-Zhang \cite{YZ15} and Andreatta, Goren, Howard, and Madapusi Pera \cite{AGHMP}.  In addition to the average case, known cases include the abelian case (\cite{Colmez}, \cite{Obus}), the quartic case (\cite{Yang10-AJM}, \cite{YaGeneral}), and  the unitary case of signature $(n, 1)$ (\cite{BHKRY}). In all known cases, the analytic side is known  and involves simple $L$-functions in a simple way. In general, the analytic side is quite mysterious and is very hard to understand. In this paper, we focus on a non-trivial special case and completely unfold the analytic side. It   sheds some light on  what kind of  $L$-functions are involved. Now let's describe the main results in a little more detail and set up some basic notation. More notation and  an introduction to the Colmez conjecture will be given in Section \ref{sect:ColmezReview}.

Let $F$ be a totally real number field of degree $n$, $\kay$ be an  imaginary quadratic field (viewed as a subfield of $\C$), and $E=F\kay$. Let $\{\tau_i\}_{i\in \Z/n}$ be the embeddings  of $F$ into $\C$. For each subset $S\subset\Z/n$ of order $0 \le r \le n$, there is a unique CM type
$\Phi_S=\{ \sigma_i:\, i\in \Z/n\}$ of $ E$ such that $\sigma_i|_F=\tau_i$ and $\sigma_i|\kay $ is the identity embedding if and only if $i\notin S$. We will say $\Phi_S$ is of signature $(n-r, r)$. In this paper, we will prove the following theorem.

\begin{theorem} \label{theo:Average1} (1) \quad When $|S|=0, 1$, $n-1$ or $n$,  the Colmez conjecture (\ref{eq:Colmez2})  holds for $\Phi_S$.

(2) \quad For any $2 \le r \le n-2$, the Colmez conjecture (\ref{eq:Colmez2}) holds for the average of all CM types of $E$ of signature $(n-r, r)$.
\end{theorem}

 As a corollary, we show in Corollary \ref{cor:SymmetricGroup} that if the Galois closure $F^c$ of $F$ has the symmetric group $S_n$ or alternating group $A_n$  as its Galois group over $\Q$, then the Colmez conjecture (\ref{eq:Colmez2}) holds for $E$, i.e., for every CM type of $E$. We mention that Barquero-Sanchez and  Masri \cite{BSM} have proved a very nice density theorem for CM number fields satisfying the Colmez conjecture   using the average Colmez conjecture and  analytic tools.

For the rest of this introduction, let $F^c$ be the Galois closure of $F$ and assume that $\Gal(F^c/\Q) =D_{2n}$ is a dihedral group. The main purpose of this paper is to compute the class function $A_S^0=A_{\Phi_S}^0$ for  every CM type $\Phi_S$ of $E$ and relate it  to characters, and  then compute the associated log derivative $Z(s, A_S^0)$ in terms of  known Hecke class characters. Both $A_S^0$ and $Z(s, A_S^0)$ are essential in the Colmez conjecture and will be defined in Section \ref{sect:ColmezReview}. The main results are Theorems \ref{theo:ZetaOdd} and \ref{theo:ZetaEven}. As a consequence, we will prove in Section \ref{sect:Colmez} the following theorem.

\begin{theorem} \label{theo1.2} Assume that $\Gal(F^c/\Q) \cong D_{2n}=\langle \sigma, \tau:\,  \sigma^n=\tau^2=1, \sigma \tau =\tau \sigma^{-1}\rangle$ with $n=[F:\Q]$. Let  $E^c=F^c\kay$, and let  $\kay_\tau=(F^c)^\sigma$ be the real quadratic subfield of $F^c$ fixed by $\sigma$.  The following are equivalent.
\begin{enumerate}
\item The Colmez conjecture (\ref{eq:Colmez}) holds for  all  class characters $\nu_k$ of the real quadratic field $\kay_\tau$,  $1 \le k \le \frac{n-1}2$, whose associated Artin characters are
     $$
     \nu_k: \Gal(E^c/\kay_\tau)\rightarrow \C^\times, \quad \nu_k(\sigma)=e^{\frac{2 \pi i k}n}, \nu_k(\rho) =-1.
     $$
     Here $\rho$ is the complex conjugation on $E^c$.

  \item The Colmez conjecture   (\ref{eq:Colmez2}) holds for all CM types $\Phi_S$ of $E$.

\item The Colmez conjecture (\ref{eq:Colmez2}) holds for  all CM types $\Phi_S$ with $|S|=2$.
\end{enumerate}
\end{theorem}

In short, the Colmez conjecture for $E$ amounts to understanding the log derivatives of some simple (CM) Hecke class characters of the real quadratic field $\kay_\tau$ in terms of geometry. This is very similar to understanding Jacobi sum characters via the Fermat curves (or vice versa).

The paper is organized as follows. In Section \ref{sect:ColmezReview}, we set up notation about CM types, class functions, and review  Colmez conjectures (\ref{eq:Colmez}) and (\ref{eq:Colmez2}). We also give a simple reinterpretation of the proved Average Colmez Conjecture. In Section \ref{sect:Dihedral}, we collect and prove some basic facts about CM types and class functions related to the CM number field $E=F\kay$ assuming that $\Gal(F^c/\Q)\cong D_{2n}$. In Section  \ref{sect:ClassFunction}, we do the main calculation for $A_S^0$ and $Z(s, A_S^0)$. To make the exposition clearer, we divide it into two cases depending on whether $n$ is odd or even. In the last section, we prove the above theorems and also discuss the Colmez conjectures for CM number field of low degrees.

\textbf{Acknowledgment.}  We thank  J. Fresan for sharing his idea to use the average Colmez conjecture to derive the Colmez conjecture  with one of the authors (T.H.). We thank the anonymous referee, Ben Howard, and Solly Patenti for making the exposition more readable and correcting numerous English errors.  Part of the work was done when  both of us visited the Morningside Center of Mathematics during summer 2016. We thank the MCM  for providing us with excellent working condition and thank Ye Tian for all his help.

\section{The Colmez Conjecture} \label{sect:ColmezReview}

In this section, we give a brief description of the Colmez conjecture  and known results after setting up notation. We also record an observation which will be used later in this paper and should be of independent interest. We refer to \cite{Colmez} and \cite{Yang10-CJM} for more detail on the Colmez conjecture.

  We fix an embedding $\bar\Q \subset \C$, and let $\Q^{\CM}$ be the union of all CM number fields in $\bar\Q$, and  let  $G^{\CM}=\Gal(\Q^{\CM}/\Q)$. It has a unique complex conjugation $\rho$ which  is in the center of $G^{\CM}$. Let $H(G^\CM)$ be the  Hecke algebra of $G^\CM$, i.e., the ring  (without identity) of locally constant functions $f$ on $G^\CM$ with values in $\C$ and  multiplication given by  the convolution:
    \begin{equation}
       f_1*f_2 (g) =  \int_{G^\CM}f_1(h)f_2(h^{-1} g) dh.
       \end{equation}
Here we require $\vol(G^{\CM})=1$. Let $H^0(G^\CM)$ be the subring of  $H(G^{\CM})$ of locally constant class functions  $f$ on $G^\CM$ ( i.e., $f(g) =f(hgh^{-1})$).  It is well-known that $H^0(G^\CM)$ has a $\C$-basis of Artin characters(characters of Gaois representations).
For a class function $A=\sum a_\chi \chi \in H^0(G^\CM)$ expressed as the linear combination of Artin characters, we define
$$
Z(s, A)=\sum a_\chi Z(s, \chi),  \quad Z(s, \chi) =\frac{L'(s, \chi)}{L(s,\chi)} +\frac{1}2 \log f_{Art}(\chi),
$$
where $f_{Art}(\chi)$ is  the Artin conductor of $\chi$.

Let $H^{00}(G^{\CM})$ be the subalgebra of $H^0(G^\CM)$ of class functions $A$ such that $A(g) + A(\rho g)$ is independent of $g \in G^{\CM}$. According to \cite[Chapter I, Section 3]{Tatebook}, $H^{00}(G^{\CM})$ has a $\C$-basis  consisting of the Artin characters $\chi=\chi_\rho$ of Artin
representations of $G^{\CM}$ such that $L(0, \chi)\ne 0$. In  particular,  for a class function $A\in  H^{00}(G^\CM)$, $Z(0, A)$ is well-defined.

On the geometric side, a CM type $\Phi$ is a function on $G^{\CM}$ with values in $\{0,1\}$ such that $\Phi(g) +\Phi(\rho g) =1$. There is a CM number field $E$ such that $\Phi(g) =1$ for all $g \in G_E^{\CM}=\Gal(\Q^{\CM}/E)$. In such a case,
$$
\{ \sigma: E\hookrightarrow \bar\Q\mid\,  \Phi(g) =1 \hbox{ for some } g \in G^{\CM} \hbox{ with } g|E=\sigma\}
$$
is a CM type of $E$ in the classical sense, which we also denote by $\Phi$. When we need to distinguish the field $E$, we will write $\Phi_E$ instead of simply $\Phi$. One can  easily associate a CM type function $\Phi$ to a classical CM type of $E$  in the obvious way. We will not distinguish the two and often  write it as a formal sum
$$
\Phi=\sum_{\sigma \in \Phi} \sigma.
$$
We may assume that $E$ is Galois over $\Q$ because of the above identification. Let
$$
\tilde\Phi=\sum_{\sigma \in \Phi} \sigma^{-1}, \quad A_{\Phi}=\frac{1}{[E:\Q]} \Phi \tilde\Phi =\sum a_g g ,
$$
and let
\begin{equation} \label{eq:A0}
A_\Phi^0=\frac{1}{|G|} \sum_{g \in G } gA_\Phi g^{-1} = \sum a_g^0 g, \quad a_g^0 = \frac{1}{|[g]|} \sum_{g_1 \in [g]} a_{g_1}
\end{equation}
be the projection of $A_\Phi$ onto $H^0(\Gal(E/\Q))\subset H^0(G^\CM)$, where $[g]$ is the conjugacy class of $g$ in $\Gal(E/\Q)$. Notice that $A_\Phi^0\in H^{00}(G^{\CM})$.  In terms of functions, one has $$
A_\Phi= \Phi * \Phi^\vee,   \quad \Phi^\vee(g) =\Phi(g^{-1}),
$$
and the projection is given by
$$
f^0(g) = \int_{G^\CM} f(h g h^{-1})dh
$$
for a locally constant  function $f \in H(G^{\CM})$.

On the other hand, let $(A, \iota)$ be a CM abelian variety of CM type $(\OO_E, \Phi)$ (with $\OO_E$-action), where $\OO_E$ is the ring of integers of $E$. Then there is a number field $L$ over which $(A, \iota)$ is defined and has good reduction everywhere. Let $\mathcal A$ be the Neron model of $A$ over $\OO_L$, and let
$\omega_{\mathcal A/\OO_L} = \wedge^n \Omega_{\mathcal A/\OO_L}$ which is an invertible $\OO_L$-module. Here $n=\frac{1}2 [E:\Q]$.  Let $0\ne \alpha \in  \omega_{\mathcal A/\OO_L}$, then  the Faltings height of $A$ is defined to be (the normalization here is the same as in \cite{Colmez} but different from \cite{Yang10-CJM})
$$
h_{\Fal}(A) = -\frac{1}{2[L:\mathbb Q]}
\sum_{\sigma: L \hookrightarrow \mathbb C}
 \log \left|\int_{\sigma(A)(\mathbb C)} \sigma(\alpha) \wedge
\overline{\sigma(\alpha)}\right| + \log |\omega_{\mathcal A/\OO_L}/\OO_L\alpha |.
$$
Colmez proved in \cite{Colmez} that
$$
h_{\Fal}(\Phi) = \frac{1}{[E:\Q]} h_{\Fal}(A)
$$
is not only independent of the choice of $A$ but also independent of the choice of $E$,  i.e., it  depends  only on the CM type $\Phi$ as a function. It is also clear that $h_{\Fal}(\Phi) =h_{\Fal}(\sigma \Phi)$ for any $\sigma \in \Gal(\bar\Q/\Q)$.

A key result of Colmez in \cite[Theorem 0.3]{Colmez} is that there is a unique linear function on $H^{00}(G^{\CM})$, called the height function $h$, such that
\begin{equation}\label{eq:Colmez1}
h_{\Fal}(\Phi) = -h(A_{\Phi}^0).
\end{equation}
 The Colmez conjecture asserts that for every class function $ f \in H^{00}(G^\CM)$, one has
 \begin{equation} \label{eq:Colmez}
 h(f) = Z(0, f^\vee).
 \end{equation}
 We remark that for a CM type $\Phi$, $(A_\Phi^0)^\vee=A_\Phi^0$. In particular, one should have
 \begin{equation}\label{eq:Colmez2}
 h_{\Fal}(\Phi) = - Z(0, A_{\Phi}^0).
 \end{equation}
Colmez proved his conjecture (\ref{eq:Colmez}) up to a multiple of $\log 2$  when $f$ is a Dirichlet character. The $\log 2$ part  was then  taken care  of by Obus \cite{Obus}. In other words, they proved the Colmez conjecture (\ref{eq:Colmez2}) when the CM field is abelian over $\Q$.  One of the authors proved (\ref{eq:Colmez2}) for quartic fields under some minor condition---the first non-abelian case \cite{Yang10-AJM}, \cite{YaGeneral}---using arithmetic intersection theory. The average version of (\ref{eq:Colmez2}) was recently proved independently in \cite{AGHMP} and \cite{YZ15}. It asserts: For a CM number field $E$ of degree $2n$ with maximal totally real subfield $F$,  one has
\begin{equation} \label{eq:AverageColmez}
\frac{1}{2^n} \sum_{\substack{\mathrm{all\ CM\ types\ of\ E}}} h_{\Fal}(\Phi) = -\frac{1}{4n} Z(0, \chi_{E/F})-\frac{1}{4} \log (2\pi),
\end{equation}
where $\chi_{E/F}$ is the quadratic Hecke character of $F$ associated to the quadratic extension $E/F$. This theorem implies that  the Colmez conjecture (\ref{eq:Colmez}) holds for all quadratic Hecke characters associated to all CM number fields. We state it as the following proposition.

\begin{proposition} \label{prop:AverageColmez}   Let $E$ be a CM number field with maximal totally real subfield $F$, and let $\chi_{E/F}$ be the Hecke character of $F$ associated to $E/F$ via the class field theory. Then the Colmez conjecture (\ref{eq:Colmez}) holds for $\chi_{E/F}$.  More precisely, view $\chi_{E/F}$ as a character of $G_F=\Gal(\Q^{\CM}/F)$ which factors through $\Gal(E/F)$, and let $f_{E/F}$ be the character of the induced representation $\Ind_{G_F}^{G^{\CM}} \chi_{E/F}$.
 Then the Colmez conjecture (\ref{eq:Colmez}) holds for $f_{E/F}$:
$$
h(f_{E/F}) = Z(0, f_{E/F}).
$$
\end{proposition}

For simplicity, we will call the Hecke characters $\chi_{E/F}$    quadratic CM characters.  The proposition asserts then that the Conjecture (2.4) holds for all quadratic CM characters.

 \begin{proof} Direct calculation (see for example \cite[(9.3.1)]{AGHMP}) gives
 $$
 \frac{1}{2^n} \sum_{\substack{\mathrm{all\ CM\ types\ of\ E}}} A_{\Phi}^0 =\frac{1}4 \mathds{1}_{G^\CM} + \frac1{4n} f_{E/F}.
 $$
 Notice that our $A_{\Phi}^0$ is $\frac{1}{[E:\Q]}a_{E,\Phi}^0$ in their notation. Here $\mathds{1}_G$ stands for the characteristic function of $G$ as a subset of $G^{\CM}$. So (\ref{eq:AverageColmez}) and (\ref{eq:Colmez1}) imply that
 $$
 \frac{1}4 h(\mathds{1}_{G^\CM}) + \frac{1}{4n} h(f_{E/F}) = \frac{1}{4n} Z(0, f_{E/F})+ \frac{1}{4} \log (2\pi).
 $$
 On the other hand,  since $\mathds{1}_{G^\CM}$ corresponds to the trivial representation of $G^{\CM}$, one has
 $$
 h(\mathds{1}_{G^\CM})= \frac{\zeta'(0)}{\zeta(0)} =\log (2 \pi).
 $$
 Therefore,
 $$
 h(f_{E/F})= Z(0, f_{E/F}).
 $$

 \end{proof}

We remark that when $F$ is Galois over $\Q$ (equivalently $E$ is Galois over $\Q$), one has
$$
f_{E/F}(g) =\begin{cases}
 n\chi_{E/F}(g)  &\ff g|F =1,
 \\
 0 &\ff g|F \ne 1.
 \end{cases}
$$
That is $f_{E/F} =n(1-\rho)$ where $n=[F:\Q]$.

\section{CM number fields of Dihedral type}  \label{sect:Dihedral}

Let $F$ be a totally real number field of degree $n$ such that its Galois closure $F^c$ (over $\Q$) is of Dihedral  type $D_{2n}$, i.e.,
$$
\Gal(F^c/\Q) = \langle  \sigma, \tau:\, \sigma^n =1, \tau^2=1, \tau \sigma \tau =\sigma^{-1} \rangle,
$$
and that $\tau$ fixes $F$.   In particular, $\{ \sigma^i:\, i \in \Z/ n\}$ gives all the real embeddings of $F$. Let $\kay$ be an imaginary quadratic field, and let $E=F\kay$ be the associated CM number field. Then the Galois closure of $E$ is $E^c=F^c\kay$. Notice that $F^c$ is totally real and $E^c$ is a Galois CM number field.    Let $\rho$ be the complex conjugation of  $E^c$. Then $\rho \in \Gal(E^c/\Q)$ and it  commutes with $\sigma$ and $\tau$, where we extend $\sigma$ and $\tau$ to $E^c$ by acting trivially on $\kay$.  Let $G^c=\Gal(E^c/\Q)$, $G=\Gal(E^c/\kay)\cong \Gal(F^c/\Q)$, and $H=\Gal(E^c/(E^c)^\sigma) \cong \Gal(F^c/(F^c)^\sigma)=\langle \sigma \rangle$. Then
$$
G^c=\langle \sigma, \tau, \rho \rangle = G \times \langle \rho \rangle.
$$

For $g=\rho, \tau, \rho\tau$, we will denote $\kay_g$ for the quadratic subfield of $E^c$ such that $\Gal(\kay_g/\Q) =\langle g \rangle$.  So $\kay_\rho=\kay$, and $\kay_\tau=(F^c)^\sigma$. The following diagram is useful in tracking various subfields of $E^c$ used  in this paper.

$$
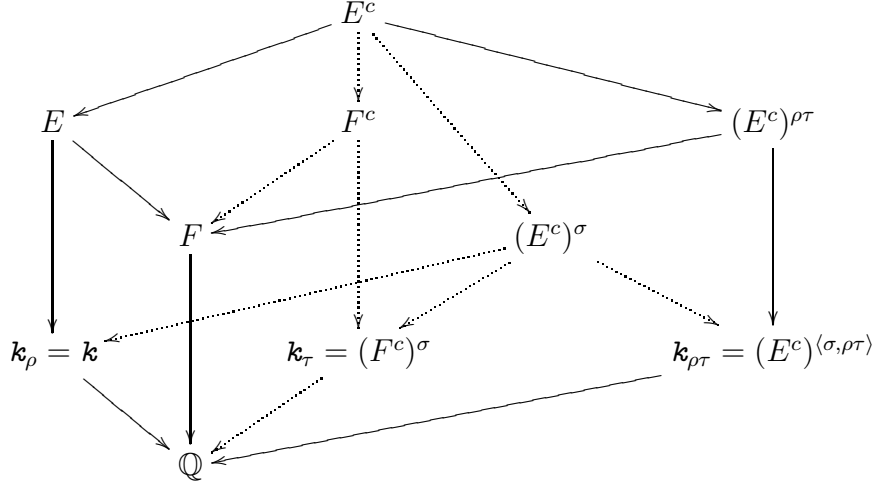
\xymatrix{&&E^c\ar[dll]\ar@{.>}[d]\ar@{.>}[ddr]\ar[drr]&& \\ E\ar[dd]\ar[dr]&&F^c\ar@{.>}[dl]\ar@{.>}[dd]&&(E^c)^{\rho\tau}\ar[dlll]
\ar[dd]\\ &F\ar[dd]&&(E^c)^\sigma\ar@{.>}[dr]\ar@{.>}[dlll]\ar@{.>}[dl]&\\
\kay_\rho=\kay\ar[dr]&&\kay_\tau=(F^c)^{\sigma}\ar@{.>}[dl]&&\kay_{\rho\tau}=(E^c)^{\langle\sigma,\rho\tau\rangle}\ar[dlll]\\ &\Q&&&}$$

\begin{lemma} For a subset $S$ of $\Z/n$, let
$$
\Phi_S=\{ \sigma^i \rho:\, i \in S\} \cup \{ \sigma^i:\,   i \notin S\} .
$$
Then $S \mapsto \Phi_S$ gives a one-to-one correspondence  between the subsets of $\Z/n$ and  CM types of $E$.
\end{lemma}

We call a CM type $\Phi$ of signature $(n-r, r)$ if $\Phi=\Phi_S$ with $|S|=r$. Two CM types $\Phi_1$ and $\Phi_2$ are equivalent, denoted by $\Phi_1 \cong \Phi_2$  if there is some $g \in \Gal(E^c/\Q)$ such that $g \Phi_1 =\Phi_2$. It is well-known that the Colmez conjecture depends only on the equivalence class of  a CM type.

\begin{lemma} \label{lem3.2}  The following are true.
\begin{enumerate}

\item One has $\Phi_S \cong \Phi_{\bar S}$, where $\bar S= \Z/n -S$.

\item  There is only one CM type of signature $(n-1, 1)$ up to equivalence.

\item Every CM type  of signature $(n-2, 2)$ is equivalent to some $\Phi_{\{0, i\}},  i\in \Z/n-\{0\}$.  Moreover, $\Phi_{\{0, i\}} \cong \Phi_{\{0, j\}}$ if and only if $i =\pm j$.  In particular, there are $[\frac{n}2]$ (the integral part of $\frac{n}2$) equivalence classes of CM types of signature $(n-2, 2)$.

\item  For $q \ge 3$,
every  CM type of signature $(p, q)$ is equivalent to  some  $\Phi_{\{0, S\}}$ where $S$ is a subset of $\Z/n -\{0\}$ of order $q-1$. Moreover, $\Phi_{\{0, S_1\}} \cong \Phi_{\{0, S_2\}}$ if and only if at least one of the following conditions hold:

(a) \quad There is some $i \in S_1$ such that $\{ -i, S_1-i\} =\{ 0, S_2\}$. Here $S-i =\{ j-i: \, j \in S\}$, and $\{ 0, S\} =\{0\} \cup S$.

(b)  \quad $S_1 =-S_2$.

\end{enumerate}

\end{lemma}
\begin{proof} Claim (1) is clear as $\rho\Phi_S =\Phi_{\bar S}$. Claim (2) is also clear as $\sigma^i\Phi_{\{0\}} =\Phi_{\{i \}}$. For (3), the same argument shows that every $\Phi_S$ with $|S|=2$ is equivalent to some $\Phi_{\{0, i\}}$ via some $\sigma^j$. Since $G^c=\langle \sigma, \tau, \rho\rangle$, $\Phi_{\{0, i\}} \cong \Phi_{\{0, j\}}$ if and only if there is some $g=\sigma^k \tau^l$ (assuming $n\ne 4$) such that $\{ g,  g \sigma^i\} =\{ 0, \sigma^j\}$. This implies $i=\pm  j\pmod n$.  The same argument also gives (4) and we leave the details to the reader.

\end{proof}

Recall that for a CM type $\Phi$, the reflex field $E^\Phi$ of $E$ is  the subfield of $\C$ (also $E^c$) generated by
$$
\tr_\Phi(z) =\sum_{g \in \Phi} g(z), \quad z \in E.
$$
The following proposition should be of independent  interest although it will not be used in this paper.

\begin{proposition} Let $\Phi=\Phi_{\{0, i\}}$ be a CM type of $E$.

\begin{enumerate}

\item If $n >4 $ and $i \ne n/2$, $E^\Phi= (E^c)^{\sigma^i\tau}$, the fixed field of $E^c$ under $\sigma^i\tau$.

\item  If $n > 4$ is even and $i =n/2$,  then  $E^\Phi= (E^c)^{\langle \sigma^i, \tau \rangle}$.

\item For $n=4$, the reflex field of $\Phi_{\{0, 1\}}$ is $(E^c)^{\langle \sigma\tau, \sigma^2 \rho \rangle}$, and the reflex field of $\Phi_{\{0, 2\}}$ is
$(E^c)^{\langle \sigma^2, \tau \rangle }$.
\end{enumerate}
\end{proposition}
\begin{proof} Let $\Phi^c$ be the extension of $\Phi_{0, i}$ to $E^c$. Then the reflex field of $\Phi$ is the same as that of $\Phi^c$. It is easy to check
\begin{align*}
\tr_{\Phi^c}(z)&=\tr_{E^c/\kay}(z) + (1+\sigma^i)(1+\tau) (\rho-1)(z).
\end{align*}
In particular, for  $z \in \kay,  \tr_{\Phi^c}(z) =2n z + 4(\bar z -z)$.  This implies that $\kay \subset E^\Phi$ when $n\ne 4$.
Assuming $\kay \subset E^\Phi$. Then it is clear from the above identity  that
an element $g \in \Gal(E^c/\kay )=\Gal(F^c/\Q)$ fixes $E^\Phi=(E^c)^{\Phi^c}$ if and only if
$$
g(1+\sigma^i)(1+\tau) (\rho-1)  = (1+\sigma^i) (1+\tau)(\rho-1).
$$
Now we check it case by case. Recall that $\Gal(F^c/\Q)=\{ \sigma^j, \sigma^j \tau:\,   0 \le j \le n-1\}$. Taking $z=x \delta$ with $x \in F^c$ and $\delta \in \kay$ with $\rho(\delta)=-\delta$,  the above identity becomes over $F^c$
\begin{equation} \label{eq3.1}
g(1+\sigma^i)(1+\tau)= (1+\sigma^i)(1+\tau).
\end{equation}
When $g=\sigma^j$, (\ref{eq3.1}) becomes (restricting on $F$ is enough)
$$
\sigma^j + \sigma^{i+j} = 1 +\sigma^i,
$$
which is possible only when $n$ is even and $i=j=\frac{n}2 \pmod n$.
When $g=\sigma^j\tau$, (\ref{eq3.1})  becomes
$$
\sigma^j + \sigma^{j-i} = 1 +\sigma^i,
$$
which is the same as $j=i$ (any case) or $i=\frac{n}2$, $j=0$ and $n$ is even. This verifies the proposition except the case $n=4$.

 Now assume that $n=4$. Write $z=x+y \delta \in E$ with  $x , y \in F$, one sees that
 $$
 \tr_{\Phi}(z) = \tr_{F/\Q}(x) + \tr_{F/\Q}(y) \delta -2 (y+\sigma^i(y)) \delta.
 $$
So the reflex field $E^\Phi$ is generated by
$$
\alpha(y) = \tr_{F/\Q}(y) \delta -2 (y+\sigma^i(y)) \delta,  \quad y \in F.
$$
When $i =2$,  it is easy to see $E^\Phi \subset (E^c)^{\langle \sigma^2, \tau \rangle }$.  If the inclusion is proper, $E^\Phi$ has to be fixed by at least one  of $\rho$, $\sigma$ and $\tau\rho$. One verifies that none of these elements fix $E^\Phi$.

When $i=1$, one checks that $E^\Phi \subset (E^c)^{\langle \sigma^2\rho , \sigma\tau \rangle }$. If the inclusion is proper, $E^\Phi$ has to be fixed by at least one  of $\rho$, $\sigma^2$ and $\tau\rho$. One verifies that none of these elements fix $E^\Phi$.

\end{proof}

\subsection{Artin representations and class functions}

Let $\zeta_n=e(1/n)=e^{\frac{2 \pi i}n}$. For  $j \in \Z/n$, we define
\begin{align*}
\mu_j: &H \rightarrow \C^\times, \quad \mu_j(\sigma) = \zeta_n^j, \quad \pi_j =\Ind_{H}^G \mu_j.
\end{align*}
When $j=0$,  $\pi_0 \cong \chi_0 \oplus \chi_s$ decomposes into the direct sum of two characters of $G$, factoring through $\langle \tau \rangle$. When $0 < j < n/2$ , $\pi_j$ is  an irreducible $2$-dimensional representation of $G$ with character $\chi_j =\chi_{\pi_j}$. We extend $\mu_j$ trivially to $H^c=H\times \langle \rho \rangle$,  $\pi_j^c=\Ind_{H^c}^{G^c} \mu_j^c$, and $\chi_j^c = \chi_{\pi_j^c}$.
When $n=2m$ is even, $\pi_m$ decomposes into one dimensional representations (characters): $\pi \cong \chi_{m, 0} \oplus \chi_{m, 1}$. Here ($i=0, 1$)
$$
\chi_{m, i} (\sigma) =-1, \quad \chi_{m, i} (\tau) =(-1)^i.
$$

  Let $\psi_1$ and  $\psi_{\rho}$ be the trivial and non-trivial characters of  $\Gal(\kay/\Q)$, which can also be viewed as characters of $G^c$ and even $G^{\CM}$. Given a function $f_1$ on $G$ and a  function $f_2$ on $\Gal(\kay/\Q)$,  we denote $f_1 f_2$ for the class function on $G^c =G \times \langle \rho \rangle$ in the obvious way:
$$
(f_1f_2)(g_1g_2)= f_1(g_1) f_2(g_2),  \quad g_1 \in G, g_2 \in \langle \rho \rangle.
$$
We  shorten $f_1 = f_1 \mathds{1}_{1}$, where $\mathds{1}_1(g_2) =1$ or $0$ depending on whether $g_2 =1$ or $\rho$. We also  use similar convention for a function $f_2$ on $\langle \rho \rangle$.

The following lemma is well-known and can be checked easily.

\begin{lemma} Let $m=[\frac{n}2]$ be the integer part of $\frac{n}2$.  Then
\begin{enumerate}
\item  When $n$ is odd, $\{ \chi_0, \chi_s,  \chi_j,  0< j \le m\}$ gives the  complete set of irreducible characters of $G=D_{2n}$.

\item When $n$  is even, $ \{ \chi_0, \chi_s,  \chi_j,  0< j \le m-1, \chi_{m, 0}, \chi_{m, 1} \}$ gives the  complete set of irreducible characters of $G=D_{2n}$.

\item The complete set of irreducible characters of  $G^c=G \times \langle \rho \rangle$  is the union of $\mathcal C \psi_1$ and $\mathcal C \psi_\rho$, where $\mathcal C$ is the complete set of irreducible characters of $G=D_{2n}$.

\end{enumerate}

\end{lemma}

\begin{lemma} \label{lem:character1} As class functions of $G^c$ (also of $G^{\CM}$),  one has
\begin{enumerate}
\item
\begin{align*}
2\chi_0&=2\tr_{E^c/\kay} =\chi_{\Ind_{G}^{G^c}\chi_0},
\\
2\chi_s &=\chi_{\Ind_{G}^{G^c}\chi_s},
\\
2 \chi_j &= \chi_{\Ind_H^{G^c} \mu_j},  \quad 1 \le j \le m,
\\
2 \chi_{m, i}&=\chi_{\Ind_{G}^{G^c}\chi_{m, i}}, \quad \hbox{when $n$ is even}.
\end{align*}

\item  One has
\begin{align*}
(1-\rho) \chi_0&=(1-\rho)\tr_{E^c/\kay} =\chi_{\kay/\Q},
\\
(1-\rho) \chi_s &=\chi_{\kay_{\rho\tau}/\Q},
\\
(1-\rho)  \chi_j &= \chi_{\Ind_{H^c}^{G^c} \mu_j^c \psi_\rho},  \quad 1 \le j \le m,
\\
(1-\rho)\chi_{m, i}&=\chi_{m, i } \psi_\rho, \quad \hbox{when $n$ is even}.
\end{align*}
Here $\chi_{m, i} \psi_\rho$ is viewed as a character of $G^c$ in the obvious way.
\end{enumerate}
\end{lemma}
\begin{proof} All follow from  a standard calculation. We leave it to the reader.

\end{proof}

On the other hand, when $n=2m+1$ is odd, $G$ has $m+2$ conjugacy classes
\begin{equation} \label{eq:OddOrbit}
\OO_0=\{1\},  \quad \OO_j=\{\sigma^j,\sigma^{-j}:  1\le j \le m \} \quad \hbox{ and } \OO_{m+1}=\{ \sigma^j \tau:   j \in  \Z/n\}
\end{equation}
which we also write as formal sums sometimes. We denote $\mathds{1}_{\OO_j}$ for their characteristic functions.  When  $n=2m$ is even, $G$ has $m+3$ conjugacy classes $\OO_0$, $\OO_j,  1 \le j \le m-1$ as above, $\OO_{m} =\{ \sigma^m\}$, and
$$
\OO_{m+1, i} = \{ \sigma^{i+2j}\tau:  0\le j \le m-1 \}, \quad i =0, 1.
$$
In this case, we denote $\mathds{1}_{m+1} = \mathds{1}_{\OO_{m+1, 0}}+ \mathds{1}_{\OO_{m+1, 1}}$.

\begin{lemma} \label{lem:character2}  One has
\begin{align*}
2n(1-\rho)\mathds{1}_{\OO_0} &= \chi_{\Ind_{\Gal(E^c/F^c)}^{G^c} \chi_{E^c/F^c}},
\\
f_{E/F}&=\begin{cases}
 (1-\rho)(n\mathds{1}_{\OO_0} +\mathds{1}_{\OO_{m+1}}) &\ff n \hbox{ is  odd},
 \\
   (1-\rho)(n\mathds{1}_{\OO_0} +2\mathds{1}_{\OO_{m+1, 0}}) &\ff n \hbox{ is  even}
   \end{cases}
\end{align*}
where $f_{E/F}$ is the  character of $\Ind_{\Gal(E^c/F)}^{G^c} \chi_{E/F}$  defined in Proposition \ref{prop:AverageColmez}.
\end{lemma}
\begin{proof} We verify the second  identity and leave the first one to the reader.  Let $f$ be the function defined in the right hand side.  Notice that $\pi=\Ind_{\Gal(E^c/F)}^{G^c}\chi_{E/F}$ has a standard basis $f_{j}$ where  $f_{j}(\sigma^i)=\delta_{i,j}$.  One checks  for $a=0, 1$ and $b\in \Z/n$
$$
\pi(\rho^a \sigma^b ) f_j(\sigma^i) =f_j(\rho^a \sigma^{i+b})  =(-1)^a \delta_{j, i+b}=(-1)^a f_{j-b}(\sigma^i),
$$
and so $\chi_\pi(\rho^a \sigma^b) = (-1)^a n \delta_{0, b}=f(\rho^a \sigma^b)$.  Next,
$$
\pi(\rho^a \sigma^b \tau) f_j(\sigma^i) = f_j(\rho^a \sigma^{i+b} \tau)=(-1)^a f_j(\tau \sigma^{-i -b}) =(-1)^a  \delta_{j, -i-b}.
$$
So
$$
\pi(\rho^a \sigma^b \tau) f_j= (-1)^a f_{-j-b},
$$
and
$$
\chi_{\pi} (\rho^a \sigma^b \tau) = (-1)^a |\{ j \in \Z/n\} :\,  2j =-b\}| =f(\rho^a \sigma^b \tau).
$$
\end{proof}

\begin{corollary} When $n$ is odd, one has
$$
L(s, \chi_{E/F})^2 L(s, \chi_{\kay_{\rho\tau}/\Q}) =L(s, E^c/F^c) L(s, \chi_{\kay/\Q}).
$$
\end{corollary}
\begin{proof} One checks (see Proposition \ref{prop:ClassFunctionRelation}):
$$
\mathds{1}_{\OO_{m+1}}= \frac{1}{2} (\chi_0-\chi_s).
$$
So Lemma \ref{lem:character2}  implies  (recall that we shorten $1-\rho =(1-\rho) \mathds{1}_{\OO_0}$)
$$
2f_{E/F} = 2n (1-\rho) \mathds{1}_{\OO_0} +  (1-\rho)\chi_0 -(1-\rho)\chi_s.
$$
Therefore one has by Lemma \ref{lem:character1}
$$
L(s, \chi_{E/F})^2 = L(s, \chi_{E^c/F^c}) \frac{L(s, \chi_{\kay/\Q})}{L(s, \chi_{\kay_{\rho\tau}/\Q})}.
$$
\end{proof}

\section{Class functions associated to CM types} \label{sect:ClassFunction}

In this section, we will compute the class function $A_{S}^0 =A_{\Phi_S}^0$ associated to the CM types $\Phi_S$ of $E$, and the log derivatives $Z(s, A_{S}^0)$ of the associated $L$-functions. Due to the slight difference between the odd and even cases, we do it separately to make exposition clearer.

\subsection{The general decomposition in the odd case}
In this subsection, we assume that $n=2m+1$ is odd. Recall that $G^c =G \times \langle \rho \rangle $ with $G=\Gal(E^c/\kay) = \langle \sigma, \tau \rangle \cong D_{2n}$.  In this case,  $G=D_{2n}$ has $m+2$ conjugacy classes given by (\ref{eq:OddOrbit}).  The Hecke algebra $H^0(G)$ of class functions on $G$ has an inner product
 $$
 \langle f,h\rangle=\frac{1}{|D_{2n}|}\sum_{g\in D_{2n}}f(g)\overline{h(g)}.
 $$
 It has two canonical orthogonal bases, one given by characters of irreducible representations of $G$, and the other given by the characteristic functions of conjugacy classes:
\begin{eqnarray}
f_{\OO_0}&=&\sqrt{2n}\mathds{1}_{\OO_0}, \notag\\
f_{\OO_k}&=&\sqrt{n}\mathds{1}_{\OO_k},\ (1\leq k\leq m),\\
f_{\OO_{m+1}}&=&\sqrt{2}\mathds{1}_{\OO_{m+1}}. \notag
\end{eqnarray}
The characters of irreducible representations of $G$ are given in Table \ref{character}.
\begin{table}[htp!]
\caption{character table of $D_{2n}$---the odd case}
\begin{tabular}{c|ccccc}\label{character}
& $\OO_0$ &$\cdots$&$\OO_k$&$\cdots$&$\OO_{m+1}$\\
\hline
$\chi_0$&1&$\cdots$&1&$\cdots$&1\\
$\chi_1$&2&$\cdots$&$2\cos\frac{2\pi k}{n}$&$\cdots$&0\\
$\vdots$&$\vdots$&$\cdots$&$\vdots$&$\cdots$&$\vdots$\\
$\chi_j$&2&$\cdots$&$2\cos\frac{2\pi j k}{n}$&$\cdots$&0\\
$\vdots$&$\vdots$&$\cdots$&$\vdots$&$\cdots$&$\vdots$\\
$\chi_m$&2&$\cdots$&$2\cos\frac{2\pi m k}{n}$&$\cdots$&0\\
$\chi_s$&1&$\cdots$&1&$\cdots$&-1
\end{tabular}
\end{table}
Let $A$ be the matrix in the Table \ref{character}, $\chi$ be
the column vector $(\chi_0,\cdots,\chi_j,\cdots,\chi_s)^T$,  and $f$  be the column vector $(f_{\OO_0},\cdots,f_{\OO_{m+1}})^T$. Then  $\chi=ABf$, where $B$ is the diagonal matrix with diagonal elements $(\frac{1}{\sqrt{2n}},\frac{1}{\sqrt{n}},\cdots,\frac{1}{\sqrt{n}},\frac{1}{\sqrt{2}})$. Then $AB$ is an orthogonal matrix. So, $f=BA^T\chi$. Then $\OO=B'A^T\chi$, where $\OO$ is the column vector $(\mathds{1}_{\OO_0},\cdots,\mathds{1}_{\OO_{m+1}})^T$ and $B'$ is the diagonal matrix with diagonal elements $(\frac{1}{2n},\frac{1}{n},\cdots,\frac{1}{n},\frac{1}{2})$. So we  have proved the following proposition  (note $\chi_j=\chi_{n-j}$)

\begin{proposition} \label{prop:ClassFunctionRelation} Let the notation be as above. As class functions of $G\cong D_{2n}$, we have the following relations.
\begin{align*}
\mathds{1}_{\OO_0}&=\frac{1}{2n} \left[ \sum_{j=0}^{n-1} \chi_j + \chi_s\right],
\\
\mathds{1}_{\OO_k}&=\frac{1}{n}\left[\sum_{j=0}^{n-1} \zeta_n^{jk} \chi_j + \chi_s \right], \quad  1 \le k \le m,
\\
\mathds{1}_{\OO_{m+1}}&=\frac{1}{2}[\chi_0-\chi_s].
\end{align*}

\end{proposition}

\begin{proposition} \label{prop:A0} Let the notation and assumptions be as above, and let $r=|S|$. Then
\begin{align*}
A_S^0&=\frac{1}2 \tr_{E^c/\kay}-\frac{r}n (1-\rho) \tr_{E^c/\kay}  + \frac{r}n (1-\rho)(\mathds{1}_{\OO_0} + \frac{1}n \mathds{1}_{\OO_{m+1}})
\\
&\qquad  +\frac{1}n(1-\rho)\sum_{1 \le j \le m} a_j(S)  \mathds{1}_{\OO_j} + \frac{r (r-1)}{n^2} (1-\rho) \mathds{1}_{\OO_{m+1}}.
\end{align*}
Here
$$
a_j(S)= |\{ (i, k) \in S^2:\,   i-k \equiv j \pmod n\}|.
$$
\end{proposition}
\begin{proof} Let $\Phi_S^c$ be the extension of $\Phi_S$ to $E^c$. By definition, tt is easy to check that
\begin{equation*}
\Phi_S^c=\tr_{E^c/\kay}+(\rho-1)(\sum_{i\in S}\sigma^i)(1+\tau)
\end{equation*}
and
\begin{equation*}
\tilde{\Phi}_S^c=\tr_{E^c/\kay}+(\rho-1)(1+\tau)(\sum_{i\in S}\sigma^{-i}).
\end{equation*}
Let
\begin{eqnarray*}
a&=&(\sum_{i\in S}\sigma^i)(1+\tau)+(1+\tau)(\sum_{i\in S}\sigma^{-i})
\end{eqnarray*}
and
\begin{eqnarray*}
b&=&(\sum_{i\in S}\sigma^i)(1+\tau)(\sum_{i\in S}\sigma^{-i})
\\
&=&|S| \mathds{1}_{\OO_0}+\sum_{1 \le j \le m} a_j(S) \mathds{1}_{\OO_j} +(\sum_{i\in S}\sigma^i)^2\tau.
\end{eqnarray*}
Since
$$
(\tr_{E^c/\kay})\sigma=(\tr_{E^c/\kay})\tau=\tr_{E^c/\kay}=\sum \sigma^j + \sum \sigma^j \tau,
$$
one has
\begin{align*}
A_{\Phi_S^c}&=\frac{1}{[E^c:\Q]}\Phi_S^c\tilde{\Phi}_S^c\\
&=\frac{1}{2}\tr_{E^c/\kay}+\frac{1}{4n}(\rho-1) \tr_{E^c/\kay}a+\frac{1}{n}(1-\rho)b\\
&=\frac{1}{2}\tr_{E^c/\kay}+\frac{r}{n}(\rho-1)\tr_{E^c/\kay}+\frac{1}{n}(1-\rho)b.
\end{align*}
Direct calculations give
\begin{align*}
b^0&=|S|\mathds{1}_{\OO_0} +\sum_{1\le j \le m}a_j(S) \mathds{1}_{\OO_j}+ \frac{r^2}{n}\mathds{1}_{\OO_{m+1}}.
\end{align*}
Since $1-\rho$ is a class function, one has $ ((1-\rho)b)^0=(1-\rho)b^0$. So
\begin{equation}
A_S^0 = \frac{1}{2}\tr_{E^c/\kay}+\frac{r}{n}(\rho-1)\tr_{E^c/\kay}+\frac{1}{n}(1-\rho) b^0.
\end{equation}
 The proposition is  now clear.
\end{proof}

Now we have the main formula in the odd case for the $L$-function part of the Colmez conjecture.

\begin{theorem} \label{theo:ZetaOdd} Assume that $n=[F:\Q]=2m+1$ is odd. Then for any subset $S$ of $\Z/n$ of order $r$, one has
\begin{align*}
&Z(s, A_{S}^0)
\\
&=\frac{r}{n^2} Z(s, \chi_{E/F}) + \frac{1}4 Z(s, \zeta_\kay) - \frac{r(n-r+1)}{n^2} Z(s, \chi_{\kay/\Q})
\\
&\qquad + \frac{1}{{n^2}} \sum_{\substack{1 \le j \le m \\ 1 \le k \le n-1 }} a_j(S) \zeta_n^{kj} Z(s, \mu_k^c \psi_\rho) .
\end{align*}
\end{theorem}
\begin{proof}  By Propositions \ref{prop:ClassFunctionRelation},  \ref{prop:A0},   and Lemmas \ref{lem:character1},  \ref{lem:character2}, one has
\begin{align*}
Z(s, A_S^0) &=\frac{1}4 Z(s, \zeta_\kay) -\frac{r}{n} Z(s,\chi_{\kay/\Q}) + \frac{r}{n^2} Z(s, \chi_{E/F})
\\
 &\qquad + \frac{1}{n^2} \sum_{ 1 \le j \le m} a_j(S) \left[  Z(s, \chi_{\kay/\Q}) + \sum_{k=1}^{n-1}\zeta_n^{kj}Z(s, \mu_k^c\psi_\rho) + Z(s, \chi_{\kay_{\rho \tau}/\Q})  \right]
 \\
 &\qquad + \frac{r(r-1)}{2n^2}\left[ Z(s, \chi_{\kay/\Q}) - Z(s, \chi_{\kay_{\rho \tau}/\Q}) \right].
\end{align*}
It is easy to see that
$$
\sum_{1 \le j \le m} a_j(S) = \frac{r(r-1)}2.
$$
Now  the theorem follows from simple calculation.

\end{proof}

\subsection{The even case}
In this subsection, we assume that $n=2m$ is even.  In this case,  $G$ has $m+3$ conjugacy classes given by  $\OO_0=\{1\}$, $\OO_k=\{\sigma^k,\sigma^{-k}\}$ ($1\leq k\leq m-1$), $\OO_m=\{\sigma^m\}$,  and
$\OO_{m+1,i}=\{\sigma^{i+2j}\tau:\,  0\le j \le m-1\}$, $i=0, 1$.  The irreducible characters of $G$ are given by $\chi_0, \chi_s, \chi_{m, 0}, \chi_{m, 1}$ and $\chi_j$,  $1\le j \le m-1$ as described in Section \ref{sect:Dihedral}.  For convenience, we let  $\OO_{m+1} = \OO_{m+1,0} \cup \OO_{m+1,1}$  and
$\chi_m=\chi_{m, 0} + \chi_{m, 1}$.
The character table of irreducible representations of $G$ is given as Table 2.

\begin{table}[htp]  \label{Table2}
\caption{character table of $D_{2n}$---even case}
\begin{tabular}{c|cccccc}\label{character1}
& $\OO_0$ &$\cdots$&$\OO_k, 1 \le k \le m-1$&$\cdots$&$\OO_{m+1,0}$&$\OO_{m+1,1}$\\
\hline
$\chi_0$&1&$\cdots$&1&$\cdots$&1&1\\
$\chi_1$&2&$\cdots$&$2\cos\frac{2\pi k}{n}$&$\cdots$&0&0\\
$\vdots$&$\vdots$&$\cdots$&$\vdots$&$\cdots$&$\vdots$&\vdots\\
$\chi_{j}$&2&$\cdots$&$2\cos\frac{2\pi j k}{n}$&$\cdots$&0&0\\
$\vdots$&$\vdots$&$\cdots$&$\vdots$&$\cdots$&$\vdots$&\vdots\\
$\chi_{m-1}$&2&$\cdots$&$2\cos\frac{2\pi(m-1)k}{n}$&$\cdots$&0&0\\
$\chi_{m,0}$&1&$\cdots$&$(-1)^k$&$\cdots$&1&-1\\
$\chi_{m,1}$&1&$\cdots$&$(-1)^k$&$\cdots$&-1&1\\
$\chi_s$&1&$\cdots$&1&$\cdots$&-1&-1
\end{tabular}
\end{table}

The same argument as in the odd case gives the following proposition.

\begin{proposition}  \label{prop:ClassFunctionEven} We have the following relations among class functions on $G=D_{2n}$.
\begin{align*}
\mathds{1}_{\OO_k}&=\frac{\delta_k}{n}\left[ \sum_{j=0}^{n-1} \zeta_n^{jk}  \chi_j + \chi_s\right],
\\
\mathds{1}_{\OO_{m+1,0}}&=\frac{1}{4}\left[\chi_0+\chi_{m,0}-\chi_{m,1}-\chi_s\right],
\\
\mathds{1}_{\OO_{m+1,1}}&=\frac{1}{4}\left[\chi_0-\chi_{m,0}+\chi_{m,1}-\chi_s\right].
\end{align*}
Here
$$
\delta_k =\begin{cases}
  1 &\ff  1 \le k \le m-1,
  \\
  \frac{1}2 &\ff k=0, m.
  \end{cases}
$$
\end{proposition}

\begin{proposition} \label{prop:A0Even} Let the notation and assumptions be as above, and let $r=|S|$. Then
\begin{align*}
A_{S}^0&=\frac{1}2 \tr_{E^c/\kay}-\frac{r}n (1-\rho) \tr_{E^c/\kay}  + \frac{r}{n} (1-\rho) \mathds{1}_{\OO_0}
\\
&\qquad  +\frac{r}n(1-\rho)(\sum_{1 \le j \le m-1} a_j(S)  \mathds{1}_{\OO_j}+2a_m(S)\mathds{1}_{\OO_m}) +
\frac{2}{n^2} \sum_{i=0}^1 b_i(S) (1-\rho)\mathds{1}_{\OO_{m+1,i}}.
\end{align*}
Here for $1\leq j\leq m=\frac{n}2$
$$
a_j(S)=\begin{cases}
   |\{ (i, k) \in S^2:\,   i-k \equiv j \pmod n\}| &\ff j \ne m,
   \\
   \frac{1}2  |\{ (i, k) \in S^2:\,   i-k \equiv j \pmod n\}| &\ff j =m,
   \end{cases}
$$
and for $i =0, 1$
$$
b_i(S)=|\{(j,k)\in S^2:\, j+k  \equiv i  \pmod 2\}|.
$$
\end{proposition}
\begin{proof} The same argument as in the odd case gives
\begin{eqnarray*}
A_S&=& \frac{1}2 \tr_{E^c/\kay} -\frac{r}n (1-\rho)\tr_{E^c/\kay} + \frac{r}{n} (1-\rho)\mathds{1}_{\OO_0}\\
&&+\frac{r}{n} (1-\rho) (\sum_{1 \le j \le m-1} a_j(S) \mathds{1}_{\OO_j}+2a_m(S)\mathds{1}_{\OO_m})
  +\frac{1}n (1-\rho) (\sum_{i\in S} \sigma^i)^2 \tau.
\end{eqnarray*}
Notice that every term  in the right han side  is a class function except $c=(\sum_{i\in S} \sigma^i)^2 \tau$. Direct calculation gives
$$
c^0 = \sum_{i=0}^1 \frac{2b_i(S)}{n} \mathds{1}_{\OO_{m+1, i}}.
$$
Now simple calculation proves the proposition.
\end{proof}

\begin{theorem} \label{theo:ZetaEven} Assume $n=2m$ is even and $|S|=r$. Then
\begin{align*}
&Z(s, A_{S}^0)
\\
&=\frac{r}{n^2} Z(s, \chi_{E^c/F^c}) + \frac{1}4 Z(s, \zeta_\kay) - \frac{r(2n-2r+1)}{2n^2} Z(s, \chi_{\kay/\Q})
\\
&\qquad + \frac{1}{{n^2}} \sum_{\substack{1 \le j \le m \\ 1 \le k <m  }}  a_j(S) (\zeta_n^{kj} + \zeta_n^{-kj}) Z(s, \mu_k^c \psi_\rho)
  + \frac{b_0(S)-b_1(S)}{ n^2} Z(s, \chi_{m, 0} \psi_\rho).
\end{align*}
\end{theorem}
\begin{proof}  By Propositions \ref{prop:ClassFunctionEven}, \ref{prop:A0Even}, and   Lemmas \ref{lem:character1}, \ref{lem:character2}, one has
\begin{align*}
&Z(s, A_S^0) -\frac{1}4 Z(s, \zeta_\kay) + \frac{r}{n} Z(s, \chi_{\kay/\Q}) -\frac{r}{n^2} Z(s, \chi_{E^c/F^c})
\\
&=\frac{1}{n^2}\sum_{1 \le j \le m}  a_j(S) \left[ Z(s, \chi_{\kay/\Q}) + Z(s, \chi_{\kay_{\rho \tau}/\Q}) + (-1)^j( Z(s, \chi_{m, 0}\psi_\rho) +Z(s, \chi_{m, 1}\psi_\rho))\right]
\\
&\qquad
  +\frac{1}{n^2}\sum_{1 \le j \le m}  a_j(S)\sum_{\substack{1\le k \le n-1\\ k \ne m} } \zeta_n^{kj} Z(s, \mu_k^c \psi_\rho)
   \\
  &\qquad + \frac{2}{4n^2}\sum_{i=0}^1  b_i(S) \left[ Z(s, \chi_{\kay/\Q}) - Z(s, \chi_{\kay_{\rho \tau}/\Q})
  +(-1)^i( Z(s, \chi_{m, 0}\psi_\rho) - Z(s, \chi_{m, 1}\psi_\rho))
  \right].
\end{align*}
It is easy to see that the last sum is equal to
$$
\frac{r^2}{2n^2} (Z(s, \chi_{\kay/\Q}) - Z(s, \chi_{\kay_{\rho \tau}/\Q})) + \frac{b_0(S) -b_1(S)}{2n^2} ( Z(s, \chi_{m, 0}\psi_\rho) - Z(s, \chi_{m, 1}\psi_\rho)).
$$
On the other hand, notice
 that
\begin{align*}
\sum_{1 \le j \le m}  a_j(S) &= \frac{r(r-1)}2,
\\
\sum_{1 \le j \le m}(-1)^j   a_j(S) &=\frac{b_0(S) -b_1(S)}2.
\end{align*}
Putting this all together proves the theorem.

\end{proof}

\section{Proof of the main results} \label{sect:Colmez}

\subsection{The Dihedral case}  In this subsection, we prove  Theorem  \ref{theo1.2} which we restate it here for convenience.

\begin{theorem} \label{theo:Dihedral} Let $F$ be a totally real number field of degree $n$ such that its Galois closure $F^c$ is  of $D_{2n}$ type
over $\Q$.   Let $\kay$ be an imaginary quadratic field, and let $E=F \kay$. For each subset $S$ of $\Z/n$ of order $r$, let $\Phi_S$ be the associated CM type of $E$ of signature $(n-r, r)$. Then the following are equivalent.
\begin{enumerate}
\item The Colmez conjecture (\ref{eq:Colmez}) holds for  all  class characters $\mu_k^c \psi_\rho$ of the real quadratic field $\kay_\tau=(F^c)^\sigma$,  $1 \le k < \frac{n}2$.

  \item The Colmez conjecture   (\ref{eq:Colmez2}) holds for all CM types $\Phi_S$ of $E$.

\item The Colmez conjecture (\ref{eq:Colmez2}) holds for  all CM types $\Phi_{\{0, i\}}$ with $1 \le i <\frac{n}2$.

\end{enumerate}
\end{theorem}
\begin{proof} In the even case,  $\chi_{m,0} \psi_\rho$  is a quadratic Dirichlet character associated to the imaginary quadratic subfield
of $E^c$ fixed by $\sigma^2, \tau$, and $\rho\sigma$. Notice also that $\zeta_\kay(s) = \zeta(s) L(s, \chi_{\kay/\Q})$.

Assume that $(1)$ is true. All the other characters in Theorems \ref{theo:ZetaOdd} and \ref{theo:ZetaEven} are quadratic CM characters  or the trivial character, for which the Colmez conjecture holds by Proposition \ref{prop:AverageColmez}. So we have  by Theorems \ref{theo:ZetaOdd} and \ref{theo:ZetaEven} that
$$
h_{\Fal}(\Phi_S) =- h(A_S^0) = -Z(0, A_S^0).
$$
The Colmez conjecture (\ref{eq:Colmez2}) holds for all CM types $\Phi_{S}$ of $E$.

 Claim (2) obviously implies $(3)$.

Assume that $(3)$ holds.  The same argument as above shows that  for  every subset $S$ of $\Z/n$ of order $2$ (recall that $Z(s, \mu_k^c\psi_\rho) =Z(s, \mu_{-k}^c\psi_\rho)$),  one  has
\begin{equation} \label{eq5.1}
\sum_{1 \le k <\frac{n}2 } \sum_{ 1 \le j \le m} a_j(S) (\zeta_n^{jk}+\zeta_n^{-jk}) x_k =0.
\end{equation}
Here
$$
x_k=h(A(\mu_k^c \psi_\rho)) - Z(0, \mu_k^c \psi_\rho),
$$
and $A(\mu_k^c \psi_\rho)$ is the associated class function, i.e.,  the character of  the induced representation $\Ind_{H^c}^{G^c} \mu_k^c \psi_\rho$.
For $S_i=\{0, i\}$ with $1 \le i < \frac{n}2$, one has
 $a_j(S_i) =\delta_{i, j}$. Applying this to (\ref{eq5.1}), we obtain
 $$
 \sum_{1 \le k < \frac{n}2} (\zeta_n^{ik}+ \zeta_n^{-ik}) x_k=0
 $$
for $1 \le i <\frac{n}2 $. Solving this system of equations, one sees $x_k=0$, i.e., (1) holds.
\end{proof}

We believe that $(3)$ can be replaced by the following statement: for a given  $2 \le r \le m$ with $(r-1, n)=1$, the Colmez conjecture (\ref{eq:Colmez2}) holds for  all CM types $\Phi_{S}$ with $|S|=r$.

\subsection{ The general unitary case}
In this subsection, we assume that $F$ is a totally real  number field of degree $n$ and $\kay$ is an imaginary quadratic field (viewed as a subfield of $\C$). Let $E=F\kay$  be as in the beginning of the introduction. The following theorem is  a refinement of  Theorem \ref{theo:Average1}.

\begin{theorem} \label{theo:Average} (1) \quad When $|S|=0, 1$, $n-1$, or $n$, the Colmez conjecture (2.5)  holds for $\Phi_S$. In other words, the Colmez conjecture holds for CM types of $E$ of signature $(n-1,1)$ and $(1, n-1)$, and
$$
 h_{\Fal}(\Phi_S) =-\frac{1}4 Z(0, \zeta_\kay) + \frac{1}n Z(0, \chi_{\kay/\Q}) -\frac{1}{n^2} Z(0, \chi_{E/F}).
$$

(2) \quad For any $1 < r < n-1$, the Colmez conjecture holds for the average of all CM types of $E$ of signature $(n-r, r)$. That is,
\begin{align*}
\sum_{|S|=r} h_{\Fal}(\Phi_S)
&=-\sum_{|S|=r} Z(0, A_S^0)
\\
&= -\frac{1}4 \begin{pmatrix} n \\ r \end{pmatrix} Z(0, \zeta_\kay)
  + \begin{pmatrix} n -2\\ r-1 \end{pmatrix} Z(0, \chi_{\kay/\Q})
 \\
 &\qquad  -\frac{1}{n}\begin{pmatrix} n-2 \\ r-1 \end{pmatrix}  Z(0, \chi_{E/F}).
\end{align*}
\end{theorem}
\begin{proof}Let $F^c$ be the Galois closure of $F$ over $\Q$ and $E^c= F^c\kay$. Let $G^c=\Gal(E^c/\Q)$, $G=\Gal(E^c/\kay)$, $H^c=\Gal(E^c/F)$, and $H=\Gal(E^c/E)$. Then $G^c=G \times \langle \rho \rangle$ and $H^c=H\times \langle \rho \rangle$ with $\rho$ being the complex conjugation of $E^c$ (also $E$ and $\kay$). For convenience, we identify $\Z/n$ with $T_n=\{1, \cdots, n\}$ in the proof. If we take coset representatives $\{ \sigma_j\}$, then
$$
G= \bigcup_{j=1}^n H \sigma_j= \bigcup_{j=1}^n \sigma_j H,
$$
and   $\{ \sigma_j|E: 1 \le j \le  n\}$ gives the $n$ distinct complex embeddings of $E$ whose restrictions on $\kay$ is the identity map. Then  for a subset $S$ of $T_n$,
$$
\Phi_S= \tr_{E/\kay} + (\rho-1) \sum_{j \in S}  \sigma_j
$$
and its extension to $E^c$ is
$$
\Phi_S^c =\tr_{E^c/\kay} + (\rho-1) \sum_{j \in S, h\in H}  \sigma_j h .
$$
So
\begin{align*}
[E^c:\Q] A_S &= A_S^c \tilde{A}_S^c
\\
 &=( \tr_{E^c/\kay})^2 + (\rho-1) \tr_{E^c/\kay}  \sum_{j \in S, h\in H}  \sigma_j  h   + (\rho-1) \sum_{j \in S, h\in H}  \sigma_j  h\tr_{E^c/\kay}
  \\
   &\qquad + (\rho-1) \sum_{j \in S, h\in H} \sigma_j h \sum_{j \in S, h\in H} h^{-1} \sigma_j^{-1} (\rho-1)
  \\
  &=[E^c:\kay] \tr_{E^c/\kay} -2|S| |H| (1-\rho) \tr_{E^c/\kay} +2|H| (1-\rho)\sum_{h \in H} \sum_{(i, j)\in S^2} \sigma_i h \sigma_j^{-1}.
\end{align*}
So, for $|S|=r$,
\begin{equation}
A_S= \frac{1}2 \tr_{E^c/\kay}- \frac{r}{n} (1-\rho)\tr_{E^c/\kay} + \frac{1}n (1-\rho) \sum_{h \in H} \sum_{(i, j)\in S^2} \sigma_i h \sigma_j^{-1}.
\end{equation}
Let $t_r=\begin{pmatrix} n \\ r \end{pmatrix}$. Then,   for $r \ge 2$,
\begin{align*}
&\sum_{|S|=r} A_s  -\frac{t_r}2 \tr_{E^c/\kay}+ \frac{r t_r }{n} (1-\rho)\tr_{E^c/\kay}
\\
&=\frac{1}{n} (1-\rho) \sum_{h\in H} \sum_{(i, j) \in T_n} \sum_{|S|=r,  (i, j) \in S^2 } \sigma_ih\sigma_j^{-1}
\\
&= \frac{1}{n}\begin{pmatrix} n-2\\ r-2 \end{pmatrix}  (1-\rho) \sum_{h\in H} \sum_{(i, j) \in T_n, i\ne j}  \sigma_i h \sigma_j^{-1}
 + \frac{1}{n}\begin{pmatrix} n-1\\ r-1 \end{pmatrix}  (1-\rho) \sum_{h\in H}\sum_{i\in T_n} \sigma_i h \sigma_i^{-1}
 \\
&=\begin{pmatrix} n-2\\ r-2 \end{pmatrix}  (1-\rho) \tr_{E^c/\kay} + \frac{1}n \left[\begin{pmatrix} n-1\\ r-1 \end{pmatrix} -\begin{pmatrix} n-2\\ r-2 \end{pmatrix}  \right](1-\rho) \sum_{h\in H}\sum_{i\in T_n} \sigma_i h \sigma_i^{-1}
\\
 &=\begin{pmatrix} n-2\\ r-2 \end{pmatrix}  (1-\rho) \tr_{E^c/\kay}
 + \frac{1}{n}\begin{pmatrix} n-2\\ r-1 \end{pmatrix}  (1-\rho) \sum_{h\in H}\sum_{i\in T_n} \sigma_i h \sigma_i^{-1},
\end{align*}
which is already a class function. Here we use the simple fact
$$
\sum_{h\in H} \sum_{(i, j) \in T_n}  \sigma_i h \sigma_j^{-1}
 =n \sum_{h \in H} \sum_{k\in T_n} \sigma_k h =n \tr_{E^c/\kay}.
$$
 So
\begin{align*}
\sum_{|S|=r} A_S^0 = \sum_{|S|=r} A_S&= \frac{1}2 \begin{pmatrix} n\\ r\end{pmatrix} \tr_{E^c/\kay}
  - \begin{pmatrix} n-2\\ r-1 \end{pmatrix} (1-\rho) \tr_{E^c/\kay}
 \\
 &\qquad  +\frac{1}{n}\begin{pmatrix} n-2\\ r-1 \end{pmatrix}  (1-\rho) \sum_{h\in H}\sum_{i\in T_n} \sigma_i h \sigma_i^{-1} .
\end{align*}
Now we check that
\begin{equation} \label{eq6.2}
 (1-\rho) \sum_{h\in H}\sum_{i\in T_n} \sigma_i h \sigma_i^{-1}= \chi_{\Ind_{H^c}^{G^c} \chi_{E/F}}= f_{E/F}
\end{equation}
is the class function defined in Proposition \ref{prop:AverageColmez}. Let $f$ be the class function on the left hand side and  $\pi=\Ind_{H^c}^{G^c} \chi_{E/F}$. Then
$$
f(g) =
  \sum_{h \in H} \sum_{i\in T_n}\delta_{\sigma_i h \sigma_i^{-1}, g} -  \sum_{h \in H} \sum_{i\in T_n}\delta_{ \sigma_i h \sigma_i^{-1},\rho g}.
$$
Here $\delta_{a, b}$ is the Kronecker delta function. On the other hand, $\pi$ has a standard basis $\{u_{\sigma_j}\}$, where
$$
u_{\sigma_j} (\sigma_i)=\delta_{i, j}.
$$
Given $g \in G=\Gal(E^c/\kay)$, $g$ gives a transformation $j \mapsto g(j)$ on $T_n$ via
$$
\sigma_j g = h \sigma_{g(j)}  \hbox{ for some } h \in  H.
$$
Simple calculation  shows
\begin{align*}
\pi(g) u_{\sigma_i} &= u_{\sigma_j} \quad \hbox{if }  g(j) =i.
\end{align*}
 So
$$
\chi_\pi(g) = | \{ i\in T_n:\, \sigma_i g \sigma_i^{-1} \in H\} | =2 |\{(i, h) \in T_n \times H:\, g = \sigma_i h \sigma_i^{-1} \} | =f(g).
$$
On  the other hand,
$$\chi_\pi(\rho g) =\chi_{E/F}(\rho) \chi_\pi(g) =-f(g) =f(\rho g).
$$
In summary, we have verified (\ref{eq6.2}). One has by Proposition \ref{prop:AverageColmez},
$$
\sum_{|S|=r} h_{\Fal} (\Phi_S) =-\sum_{|S|=r} h(A_S^0) =-\sum_{|S|=r}  Z(0,  A_S^0).
$$
This proves (2) for $ r \ge 2$.

 For $r=1$, the same calculations (easier) give
$$
\sum_{|S|=1} A_S^0= \frac{n}2 \tr_{E^c/\kay} -(1-\rho)\tr_{E^c/\kay} + \frac{1}{n} f_{E/F}.
$$
Notice that all CM types of $E$ of signature  $(n-1, 1)$ are equivalent. So for $|S|=1$ one has
$$
h_{\Fal}(\Phi_S) =-Z(0, \Phi_{S}^0) = -\frac{1}4 Z(0, \zeta_\kay) + \frac{1}n Z(0, \chi_{\kay/\Q}) -\frac{1}{n^2} Z(0, \chi_{E/F}).
$$
As a CM type of $E$ of signature $(n-r, r)$ is  equivalent to a CM type of signature $(r, n-r)$, the case $|S|=n-1$ is also true.

Then case $r=0$ (also $r=n$) is a restatement of the well-known Chowla-Selberg formula. Indeed, $\Phi_S =\tr_{E/\kay}$ is in this case the extension of the CM type $\{ 1\}$ of $\kay$ to $E$. So
$$
h_{\Fal}(\Phi_S) = h_{\Fal}(\Phi_{\kay, \{ 1\}})=\frac{1}2 h_{\Fal}(A)=-\frac{1}4 Z(0, \zeta_\kay)
$$
by the Chowla-Selberg formula (\cite[(0.15)]{KRY1}), where $\Phi_{\kay, \{1 \}}$ is the CM type $\{1 \}$ of $\kay$, and $A$ is a CM elliptic curve with CM by $\OO_\kay$.
\end{proof}

\begin{corollary} \label{cor:SymmetricGroup}  Let $F^c$ be the Galois closure of $F$ over  $\Q$ and assume that $\Gal(F^c/\Q) \cong S_n$  or $A_n$ is the symmetric group  or the alternating group of order $n$ with $n=[F:\Q]$. Then the Colmez conjecture holds for every CM type of $E=F\kay$, where $\kay$ is an imaginary quadratic field.
\end{corollary}
\begin{proof} Let $\Phi=\Phi_S$ be a CM type of $E$ with $|S|=r$ as above. By Theorem \ref{theo:Average}, we may assume  $2 \le r \le n/2$.    In this case,  notice that $S_n$ and $A_n$ act transitively on the subsets $S$ of order $r$ of $T_n=\{1, \cdots, n\}$. So $\Gal(F^c/\Q)$ acts on the CM types $\Phi_S$ of $E$ of signature $(n-r, r)$ transitively.  Theorem \ref{theo:Average} implies that the Colmez conjecture holds for each of them.

\end{proof}

\subsection{CM number fields of small degree}  It is well-known and easy to check that the Colmez conjecture only depends on the Galois orbit of a CM type.
Fresan observed that the average Colmez conjecture  and the Colmez conjecture for abelian CM number fields (both are theorems now) imply that every quartic CM number field satisfies the Colmez conjecture. He and one of the authors (T.Y.)  showed using the same idea  in a private note  that every sextic CM number field satisfies the Colmez conjecture with the help of Dodson's classification of sextic  CM number fields and their CM types (\cite{Dodson}). We consider the case $n=5$ here.  Let $E$ be a CM number field  with maximal totally real number field $F$. Let $E^c$ be the Galois closure of $E$, and let $F^c$ be the Galois closure of $F$.  Then  $E^c$ is a CM number field and $F^c$ is a totally real subfield (which may be not maximal). Let $G= \Gal(E^c/\Q)$ as above and let $G_0 =\Gal(F^c/\Q)$  as in \cite{Dodson} (this notation  is a little different from the rest of this paper). Assume that $F=\Q(\alpha)$, and let $\alpha_1=\alpha, \alpha_2, \cdots \alpha_n$ be the Galois  conjugates of $\alpha$. Then the action of $G_0$ on these roots gives an embedding of $G_0$ into $S_n$. The following is a special case of  \cite[5.1.3]{Dodson}.

 \begin{proposition} \label{prop:Dodson}  Let the notation be as above and assume $n=[F:\Q]=5$. Then  the possible Galois groups $G$ are:
 \begin{enumerate}
 \item $G=G_0 \times \langle \rho \rangle $ with $G_0\cong  \Z/5, D_{10}, \Z/5 \rtimes \Z/4, A_5$ or $S_5$.  Here
 $$
 \Z/5 \rtimes \Z/4 \cong \langle (12345), (2354)  \rangle \subset S_5.
 $$

 \item $G =(\Z/2)^5 \rtimes G_0$. Here $G_0$  is again one of the five choices listed in (1),  and it acts on $(\Z/2)^5$ by permutation of the coordinates.   In this case, all CM types of $E$ are equivalent.
 \end{enumerate}
 \end{proposition}

\begin{theorem}\label{theo:Degree5} The Colmez conjecture (\ref{eq:Colmez2}) holds for all CM number fields of degree $10$ with  the following  possible exception: $E=F \kay$ where $F$ is a totally real number field of degree $5$ whose Galois closure $F^c$ is of type  $D_{10}$ over $\Q$ and $\kay$ is an imaginary quadratic field. In such a case, the Colmez conjecture (\ref{eq:Colmez2}) holds for $\Phi_S$ with $|S|=0, 1,  4, 5$. Moreover, every $\Phi_S$ with $|S|=2, 3$ is equivalent to $\Phi_{\{1, 2\} }$ or $\Phi_{\{1, 3\}}$. Finally, the Colmez conjecture (\ref{eq:Colmez}) holds for $A_{\{1, 2\}}^0+ A_{\{1, 3\} }^0$ in this case.
\end{theorem}
\begin{proof} It is a simple case by case verification using Proposition \ref{prop:Dodson}. In case (2), the average Colmez conjecture proved in  \cite{AGHMP} and \cite{YZ15} implies the Colmez conjecture as all the CM types of $E$ are equivalent. In case (1), there is an imaginary quadratic subfield $\kay$ of $E^c$ such that $\Gal(E^c/\kay) \cong \Gal(F^c/\Q)$, and $E=F\kay$. So every CM type of $E$ has the form $\Phi_S$ with $S \subset  T_5=\{1, 2, 3, 4, 5\}$. The Colmez conjecture holds for $\Phi_S$ for $|S|=0, 1, 3, 4$ by Theorem \ref{theo:Average}.

Now we assume that   $|S|=2$. The abelian case $G_0=\Z/5$ is known due to Colmez and Obus. In all other cases, $\Phi_S$ is clearly equivalent to some $\Phi_{\{1, i\}}$,  $i=2, 3$. In cases $G_0=A_5, S_5$ or $\Z/5 \rtimes \Z/4$, there is $\tau \in G_0$ such that $\tau(1)=1$ and $\tau(2) =3$, and so $\tau \Phi_{\{ 1, 2\}}= \Phi_{\{1, 3\}} $. So all $\Phi_S$ with $|S|=2$ are equivalent, and Theorem \ref{theo:Average} implies the Colmez conjecture for all $\Phi_S$ with $|S|=2$. The same is true for $|S|=3$ as $\Phi_S $ is equivalent to $\Phi_{\bar S}$.
 We are left with $G=G_0 \times \langle \rho \rangle $, $G_0=D_{10}$. Theorem \ref{theo:Dihedral} takes care of this case.
\end{proof}

\bibliographystyle{alpha}
\bibliography{reference}

\end{document}